\documentclass[11pt]{article}
\usepackage{amsmath}
\usepackage[dvipsnames,usenames]{color}
\usepackage{amsfonts}
\usepackage{mathrsfs}
\usepackage{amssymb}
\usepackage{amsthm}
\def\h{\hbox}
\def\<{\langle}
\def\>{\rangle}
\numberwithin{equation}{section}
\def\<{\langle}   
\def\>{\rangle}

\def\ra{\rightarrow}
\def\p{\partial}
\def\a{\alpha}
\def\g{\gamma}

\def\B{\mathcal B}
\def\O{\Omega}

\def\PP{{\mathbb P}}
\def\BB{{\mathbb B}}
\def\F{\mathcal F}

\def \sm{\setminus}
\def\-{\overline}

\def\o{\omega}
\def\ov{\overline}

\def\h{\hbox}
\def\d{\delta}

\def\d{\delta}
\def\d{\delta}
\def\b{\beta}
\def\a{\alpha}

\def\CC{{\Bbb C}}

\def\NN{{\Bbb N}}

\def\BB{{\Bbb B}}

\def\ov{\overline}

\def\ld{\lambda}
\def\O{\Omega}
\def\o{\omega}

\def\sm{\setminus}

\def\h{\hbox}

\def\wt{\widetilde}
\def\ra{\rightarrow}
\def\p{\partial}

\def\zbar{\ov z}

\def\d{\delta}

\def\a{\alpha}
\def\d{\delta}

\def\wbar{\overline{w}}
\def\Ol{\overline}

\def\d{\delta}
\def\b{\beta}
\def\dbar{\ov\partial}

\def\a{\alpha}

\def\a{\alpha}
\def\g{\gamma}
\newtheorem{theorem}{Theorem}[section]
\newtheorem{lemma}[theorem]{Lemma}
\newtheorem{corollary}[theorem]{Corollary}

\newtheorem{conjecture}[theorem]{Conjecture}

\date{\ }

\begin{document}
\title{\bf   Bergman metrics with constant holomorphic sectional curvatures}
\author{{Xiaojun  Huang}\footnote{Supported in part by  DMS-2247151}
\qquad Song-Ying Li} \maketitle
\centerline{\bf With an Appendix by John N. Treuer\footnote{Supported in part by DMS-2247175 and an AMS-Simons Travel Grant}}

\bigskip\bigskip
 \centerline{Dedicated to the memory
of Professor Joe  Kohn}

\bigskip

\begin{abstract}\vskip 3mm\footnotesize
\noindent  The paper studies complex manifolds whose Bergman metrics are incomplete but 
have constant holomorphic sectional curvature. Among   results
obtained in this paper is the proof of a long standing folklore conjecture that a Stein manifold has   a negative constant holomorphic sectional curvature  if and only if it is biholomorphic to a  ball with a   pluri-polar set  removed.  Moreover   we 
construct  a  domain in $\CC^2$  whose Bergman metric is  well-defined and  has a positive constant 
holomorphic sectional curvature, which appears to be the first example of this kind. In  Section 5  and the  Appendix provided by John Treuer,  it will be shown  that under natural assumptions, there is no complex manifold whose Bergman metric is flat.

\medskip

Abstract : Cet article  \'etudie les vari\'et\'es complexes dont la m\'etrique de Bergman est non compl\'ete mais  \'a courbure sectionnelle holomorphe constante. Parmi les r\'esultats obtenus dans cet article, nous donnons une preuve d’une conjecture folklorique de longue date et nous construisons un domaine dont la m\'etrique de Bergman a \'etonnamment une courbure sectionnelle holomorphe constante et positive.

\vskip 4.5mm

\end{abstract}

\section{Introduction}
In this paper, we study  when a complex manifold has a well defined Bergman metric  with  constant holomorphic  sectional  curvature. Bergman metrics and Bergman kernels  are among the most
important objects to study in  complex analysis and geometry. Studies on
Bergman metrics  and Bergman kernels of complex manifolds date back to the work of
Stefan Bergman \cite{Ber1} and S. Kobayashi \cite{Ko}. There are
well-known classical results due to S. Bergman \cite{Ber2}, S.
Kobayashi \cite{Ko} and Q. Lu \cite{Lu}, C. Fefferman \cite{Fe},
etc.
 Bergman showed  in \cite{Ber2} that the
holomorphic sectional curvature of the  Bergman metric of a bounded
domain in $\CC^n$ is bounded  from above by the one for projective spaces, i.e,
is bounded from above by $+2$. Kobayashi generalized this upper bound
estimate  to complex manifolds in \cite{Ko}.  Lu proved in \cite{Lu}
that a bounded domain $\O$ in $\CC^n$ with a complete Bergman metric (such
an $\O$ is then pseudoconvex by a result in \cite{Bre}) has a
constant holomorphic sectional curvature if and only if $\O$ is
biholomorphic to the ball. Fefferman \cite{Fe} studied the boundary asymptotic expansion of the Bergman kernel function  of  a smoothly bounded strongly pseudoconvex domain.  In particular, using the expansion of Fefferman,  it is proved in \cite{Kle}  that the holomorphic sectional curvature of the Bergman metric of a smoothly bounded strongly pseudoconvex domain  is asymptotically negative constant near its boundary.
Fu \cite{Fu}  showed that the holomorphic sectional curvature  of the Bergman metric of a smoothly bounded   Reinhardt domain of finite type in $\CC^2$ is bounded from above by a negative constant near its boundary. 

Bergman metrics    have been extensively
studied for almost a century  since the foundational work of Stefan
Bergman. Still, our understanding is far from being complete and any
new  discovery related to these metrics serves as an important
resource of motivation for a future investigation. In this paper, we
will provide   examples   revealing new phenomena such as the
construction of a  domain with constant positive holomorphic
sectional curvature for its Bergman metric. Previously, all examples of complex manifolds with constant holomorphic sectional curvature $\kappa$  for their Bergman metrics are such that $\kappa<0$. We will  answer a
long-standing folklore conjecture on characterizing Stein manifolds
whose Bergman metrics  have a negative constant holomorphic 
sectional curvature without assuming the metric is complete, as opposed to the classical theorem of Lu. We  formulate  naturally related
new conjectures on Bergman metrics.

More precisely,  we first study the question of when does a complex
manifold have a positive constant holomorphic sectional curvature
for its Bergman metric?  Along these lines,  our results in $\S 3$ confirm in particular the following:

\begin{theorem}\label{thm6.1}
Let $\Omega$ be a  complex manifold of dimension $n$. Suppose its
Bergman space $A^2(\O)$ is base point free and
separates holomorphic directions. Suppose that the Bergman metric of
$\O$ has positive constant holomorphic sectional curvature.  Then
$A^2(\O)$ is of finite dimension. Moreover, $\O$ is biholomorphic to
a  domain in ${\mathbb P}^n$.
\end{theorem}

Here $A^2(\Omega)$  is the Bergman space of $\Omega$ consisting of the linear space of $L^2$-integrable holomorphic $n$-forms. The other conditions on $A^2(\Omega)$ in the theorem will be explained in the next section, which   are needed to have a well-defined Bergman metric on $\Omega$.  In particular, we have the following:

\begin{corollary}
Let $\Omega$ be a  bounded domain in ${\mathbb C}^n$. Then its
Bergman metric cannot have a positive constant holomorphic
sectional curvature.  On the other hand,  there does exist an unbounded Reinhardt domain whose Bergman metric has a positively constant holomorphic sectional curvature.
\end{corollary}

Motivated by an old conjecture of  J. Wiegerinck, we formulate a new conjecture that  a Stein manifold, in particular, a pseudoconvex  domain
in $\CC^n$,  cannot have  
 its  Bergman metric positively   constantly curved.  (See Conjecture \ref{conj1}).

We then study the case where the Bergman metric has a negative constant  holomorphic sectional
curvature.  We will prove the following in $\S 4$:

\begin{theorem}\label{thm5.2}
Let $\Omega$ be a Stein manifold of complex dimension $n\ge 1$.
Assume its Bergman space $A^2(\Omega)$ is base point free and
separates holomorphic directions.
 Then its Bergman metric   has a negative constant holomorphic sectional
curvature if and only if it is biholomorphic to the unit ball in $\CC^n$ with possibly  a closed pluripolar set removed.
\end{theorem}

As an immediate consequence, we have the following:

\begin{corollary}\label{cor5.3}  Suppose $\Omega$ is a bounded pseudoconvex domain in $\CC^n$. Then it  has a negative constant holomorphic sectional curvature  if and only if $\Omega$ is biholomorphic to the unit ball in $\CC^n$ with 
a closed pluripolar set removed.
 \end{corollary}

Corollary  \ref{cor5.3}  confirms  a  folklore
conjecture that a bounded pseudoconvex domain has a negative holomorphic
sectional curvature for its Bergman metric if and only if it is
biholomorphic to  the unit ball of the same dimension with possibly
a pluripolar subset removed. In two recent papers by X. Dong and B.
Wong \cite{DW1}  \cite{ DW2}, the authors made an earlier important
progress on this problem for a bounded domain $D\subset\mathbb{C}^n$. They answered this problem under the
assumption that  the Calabi diastasis function $\frac{K(z, z)}{|K(z,
p)|^{2}}$ blows up on the boundary of the domain $D$ as a function of $z$ for a fixed $p\in D$, where $K(z,w)$ is
the Bergman kernel function  of the domain $D$ with $z,w\in D$.
With more natural assumptions, they also proved a multi-dimensional
Carath\'eodory theorem for such a biholomorphic equivalence map.
For precise statements of the results of Dong-Wong, the reader is
referred to their recent papers \cite{DW1} \cite{DW2}. Our proof of
this folklore conjecture in its full generality is quite different.
Among the main  ingredients in our proof  are the rigidity property of holomorphic mappings by  Calabi \cite{Ca},
delicate use of the
reproducing property of the Bergman kernel functions and  deep
results from pluripotential theory.
Here pseudoconvexity plays an important role as we need a result of
Pflug and Zwonek \cite{PZ} which is based on the $L^2$-extension
theorem of Ohsawa-Takegoshi on pseudoconvex domains.

In addition to above,  we also study the  zero curvature case.   A result  in the last section is stated as follows:

\begin{theorem}\label{thm7.1}
Let $\Omega$ be a  complex manifold of complex dimension $n$.
Suppose that its Bergman space $A^2(\O)$ is base point free and separates holomorphic directions.  Suppose that
the Bergman metric of $\O$
 has zero constant  holomorphic sectional curvature. Then $\O$ is biholomorphic to a domain $D$ in $\CC^n$   whose Bergman space has an orthonormal basis of the form $\{\frac{z^\a}{\sqrt{\a!}} \phi\}_{\a\in ({\NN_0})^n}$, 
 where $\NN_0$ is the collection of non-negative integers,  $\phi\in A^2(D)$
 with $\phi\not =0$ and $\|\phi\|=1$.
\end{theorem}

In the Appendix, John Treuer  provides a proof of the following:

 \begin{theorem}\label{thm7.2}
There does not exist a domain $D\subset \CC^n$ such that its Bergman space $A^2(D)$ 
has an orthonormal basis of the form $\{\frac{z^\a}{\sqrt{\a!}}\phi\}_{\a\in ({\NN_0})^n}$, where $\phi\in A^2(D)$
 with $\phi\not =0$ and $\|\phi\| =1$.
\end{theorem}


Combining the above results, one immediately has the following:

\begin{corollary}\label{thm7.3}
Let $\Omega$ be a  complex manifold of complex dimension $n$.
Suppose that its Bergman space $A^2(\O)$ is base point free and separates holomorphic directions.  Then its Bergman metric cannot have zero constant holomorphic sectional curvature.
\end{corollary}

Contrary to Corollary \ref{thm7.3}, we will, however,  construct a pseudoconvex domain in
${\mathbb C}^{n+1}$ with $n\ge 1$  which contains a complex totally geodesic
embedding of the complex
Euclidean space  ${\mathbb C}^n$ equipped with its standard flat
Euclidean metric. Hence the Bergman metric is flat along such a
complex submanifold. 

Since we are  concerned with Bergman metrics locally isometric to  a complex space form, 
rigidity and global extension of  holomorphic mappings   naturally appeared in our  work.   Rigidity and extension of a local isometric holomorphic map or more generally a holomorphic embedding  with its  image being contained in a proper real analytic subvariety  were first studied  in the celebrated paper by E. Calabi \cite{Ca}.   Holomorphic  mappings from a CR submanifold into a proper real analytic subvariety  have been extensively studied in the past several decades and have found important  applications  in   other fields of mathematics.  Along these lines, we  mention   the work done by Mok in \cite{Mo1} \cite{Mo2},  Lamel-Mir \cite{LN}, Mir-Zaitsev \cite{NZ},  Kossovskiy-Lamel-Stolovitch \cite{KLS}, Lamel-Mir-Rond \cite{LNR}, as well as many references therein.
For other work     on Bergman
kernels, Bergman metrics and related topics,   we  refer the  reader to Folland-Kohn
 \cite{FK},  Diederich-Fornaess \cite{DF},  Ohsawa \cite{Oh}, Ebenfelt-Xiao-Xu \cite{EXX}, Li \cite{Li},
Gaussier-Zimmer  \cite{GZ}, just to name a few.

The paper is organized as follows. In Section 2, we introduce some
basic notations and definitions.  In Section 3, we first provide a Reinhardt domain in $\CC^2$ whose Bergman metric has a constant
holomorphic sectional curvature $2$.
 We then prove  Theorem \ref{thm6.1}.  In Section 4, we prove Theorem \ref{thm5.2}.  In Section 5,  we present a proof of Theorem \ref{thm7.1}. In Section 6, we formulate a conjecture related the work here. In the Appendix, Treuer will provide a proof
 of Theorem \ref{thm7.2}.

\medskip

{\bf Acknowledgement}: The authors would like  to thank D. Zaitsev for  his many suggestions  to this paper. They also would like to thank N. Mok and M. Xiao for their helpful conversations during the preparation of this work.


\section{Notations and definitions}

Let
$$\ell^2=\{a=(a_1,\cdots,a_m,\cdots):\ \|a\|^2:=
\sum_{j=1}^{\infty}|a_j|^2<\infty,\ a_j\in \CC\}
$$
 be the Hilbert
space equipped with the standard $\ell^2$-norm  $\|\cdot\|$ defined above.

Let $M$ be a complex manifold of dimension $n$. (In this paper, all
complex manifolds are assumed to be connected).  A map
$F=(f_1,\cdots,f_n,\cdots): M\to \ell^2$ is called a holomorphic map
if each $f_j$ ($j\in \NN$) is a holomorphic function in $M$ and if
$\sum_{j=1}^\infty |f_j(z)|^2$ converges uniformly on any compact
subset of $M$. We denote by $X_F$ the closure of the linear span of
$F(M)$ in $\ell^2$. Apparently, $X_F$ is uniquely determined by the
germ of $F$ at any point $p\in M$.

 Define the infinite dimensional projective space $\PP^\infty:=\ell^2\sm \{0\}/\sim.$ Here, for any given  $x,y\in \ell^2\sm \{0\}$,
$x\sim y $ if and only if $x=ky$ for a
certain $k\in \CC\sm\{0\}.$ For $x=(x_1,\cdots,x_n,\cdots)\in
\ell^2\sm \{0\}$, we write $[x_1,\cdots,x_n,\cdots]$ for its
equivalence class in $\PP^\infty$, called the homogeneous coordinate
of the equivalence class of $x$. $\PP^m$ is naturally identified as
a closed subspace of $\PP^\infty$ by adding zeros to homogeneous
coordinates of  $\PP^m$.

A map ${\cal F}: M\ra \PP^\infty$  is called a holomorphic map from
$M$ into $\PP^\infty$ if there is a holomorphic representation $[F]$
of ${\cal F}$  near each $p\in M$. Namely, for each $p\in M$, there
is a small neighborhood $U_p$ of $p$ in $M$ such that ${\cal F}=[F]$
with $F=(f_1,\cdots, f_m, \cdots)$, where each $f_j$ with $j\in \NN$
is a holomorphic function in $U_p$ and for each $q\in U_p$, there is
a $k$, which may depend on $q$, such that $f_{k}(q)\not = 0$.
Moreover, $\sum_{j=1}^\infty |f_j(z)|^2$ is required  to  converge
uniformly on compact subsets of $U_p$. We write ${\mathcal X}_{\F}=[X_F]\subset
\PP^n$, the projectivization of the closed linear subspace
$X_F\subset\ell^2$, for any local holomorphic representation $F$.
 ${\mathcal X}_{\F}$ is independent of the choice of local holomorphic
representations.

The formal Fubini-Study metric, denoted by $\o_{st}$, of
$\PP^\infty$ with homogeneous coordinates $[z_1,\cdots,z_n,\cdots]$
is formally defined by
\begin{equation}\label{(1.1)}
\omega_{st}=i\p \dbar \log  \Big(\sum_{j=1}^\infty |z_j|^2\Big).
\end{equation}

Let \begin{equation}\label{(1.2)} \F: \ (M, \omega) \to (
\PP^\infty, \omega_{st})
\end{equation}
be a holomorphic map, where $\o$ is a  K\"ahler metric over $M$. If
for any  local holomorphic representation
$\F=[f_1,\cdots,f_n,\cdots]$ over $U\subset M$, it holds that
\begin{equation}\label{(1.3)}
\o=\F^*(\omega_{st}):=i\p \dbar \log \Big(\sum_{j=1}^\infty
|f_j(z)|^2\Big)\quad  \h{over } U,
\end{equation}
we call  $\F$  a local holomorphic isometric embedding from $M$ into
$\PP^\infty$.

\bigskip
Let $\Omega$ be a complex manifold of complex dimension $n\ge 1$.
Write $\Lambda^n(\Omega)$  for the space of  holomorphic $(n,
0)$-forms on $\Omega$ and define the Bergman space of $\Omega$ to be
\begin{equation}
A^2(\Omega):=\Big\{f\in
\Lambda^n(\Omega):(-1)^{\frac{n^2}{2}}\int_{\Omega} f\wedge
\overline f<\infty \Big\}.
\end{equation}
  Then $A^2(\Omega)$ is a Hilbert space with
 an inner product defined below:
\begin{equation}
(f, g)=(-1)^{\frac{n^2}{2}}\int_{\Omega}f\wedge\overline g,\quad
 ~\text{for all}~f, g\in A^2(\Omega).
 \end{equation}
  Assume that $A^2(\O)\not = 0$. Let $\{f_j\}_1^N$ be an orthonormal
 basis of $A^2(\Omega)$  with $N\le \infty$ and define the Bergman kernel $(n,n)$-form \cite{Ko} to
 be $K_\Omega=\sum_{j=1}^N f_j\wedge\overline f_j$. In a local holomorphic coordinate
  chart $(U, z)$ on $\Omega$, we have
\begin{equation}
K_\Omega=K_\Omega(z, \overline z) dz_1\wedge\cdots\wedge dz_n\wedge
d\overline z_1\wedge\cdots\wedge d\overline z_n~\text{ in }~ U.
\end{equation}

In this paper, we will assume that $A^2(\O)$ is {\it base-point
free} in the sense that $K_\Omega(z)=K_\Omega(z,z)>0$  on $\Omega$.
Equivalently, we assume for any $p\in M$, there is a
$L^2$-integrable holomorphic $n$-form $\phi$ such that $\phi(p)\not
=0$. We define a Hermitian $(1, 1)$-form on $\Omega$ by
$\omega_\Omega=\sqrt{-1}\partial\overline\partial \log K_\Omega(z,
\overline z).$ We call $\omega_\Omega$ the Bergman metric  on
$\Omega$ if it induces a positive definite tensor on $\Omega$.

When $\O\subset \CC^n$ is a domain, we identify a holomorphic
$n$-form $\o=f dz_1\wedge\cdots\wedge dz_n$ with the holomorphic function $f\in
L^2(\Omega)$. Thus $A^2(\O)$ reduces to the Hilbert space of
$L^2$-integrable holomorphic functions, which is the classically
defined Bergman space. Then the Bergman kernel function  is now a
real analytic function over $\O$.


We say that $A^2(\O)$ {\it separates holomorphic
directions} if for any $p\in M$ and a non-zero $X_p\in T_p^{(1,0)}$,
there is a $\phi\in A^2(\O)$ such that $\phi(p)=0$ and $X_p(\phi)(p)\not =0$ in a
local holomorphic chart near $p$ where $\phi$ is identified with a
holomorphic function near $p$ as described above.

Assume that the Bergman space $A^2(\O)$
 is base point free. For an
orthonormal base $\{\phi_j\}_{j=1}^{N}$, we then have a well-defined
holomorphic map $\B:=[\phi_1,\cdots,\phi_N,0,\cdots]$ for $N<\infty$
or $\B:=[\phi_1,\cdots,\phi_m,\cdots]$ for $N=\infty$ from $\O$ into
$\PP^\infty$. $\B$ is called a Bergman-Bochner map.
 We say  $A^2(\O)$ separates points  if a Bergman-Bochner
 map  ${\cal B}: \O \to \PP^\infty$ is one-to-one. Later  we will see that in the consideration interested here, the property that separating points will be a consequence of separating holomorphic directions and base point free.

 Bounded domains in $\CC^n$ are  typical examples of complex manifolds whose Bergman spaces satisfy all
 of the above three
 conditions: base-point free, separates points and separates
 holomorphic  directions.

Kobayashi \cite{Ko} showed that for a complex manifold,  the assumption that its
Bergman space  $A^2(\O)$ is base-point free and separates
holomorphic directions is the  necessary and sufficient condition that  the Bergman $(1,1)$-form of $\O$ gives  a
biholomorphcially invariant K\"ahler metric, called the Bergman
metric of $\O$.

\section{ Bergman metrics with positive constant   holomorphic sectional curvature}

In this section, we study the problem if there is a complex manifold
whose Bergman metric has   a constant holomorphic
sectional curvature (HSC).

 We start with the following  surprising  example
  which demonstrates  that the Bergman metric of
a domain may indeed have a positive constant holomorphic sectional
curvature $2$, a maximally allowable one.

 Let $\alpha>2$ and  define a complete Reinhardt domain $D(\a)$ as follows:
\begin{equation}
D(\alpha)=\Big\{(z,w)\in \CC^2: \Big| |w|^2-|z|^2\Big|<
r(z,w)^{-\alpha} \Big\},\quad r(z,w)=1+|z|^2+|w|^2.
\end{equation}
\begin{theorem} \label{example-6}For  $2<\alpha<3$, the following
statements hold:

\noindent (i) The Bergman space
\begin{equation}
A^2(D(\alpha))=\hbox{Span}\{1, z, w\}.
\end{equation}

\noindent (ii) The Bergman metric of $D(\a)$ has  constant
holomorphic sectional curvature $2$.

\end{theorem}

\begin{proof} By the definition of $D=D(\alpha)$, one can easily see that
$$
\int_D dv(z,w)\lesssim \int_D |z|^2 dv(z,w).
$$
Here we write $dv$ for the related Lebesgue measure.

To prove that $1, z\in A^2(D(\alpha))$, it suffices to compute that
\begin{eqnarray*}
\int_D |z|^2 dv(z,w)
&=&\int_{\CC}
dv(z)\int_{|z|^2-r(z, w)^{-\alpha} <|w|^2<|z|^2 +r(z,w)^{-\alpha} }
 |z|^2 dv(w)\\
&\lesssim&\int_{\CC}  |z|^2 dv(z) \int_{|z|^2-r(z, 0)^{-\alpha} <|w|^2<|z|^2 +r(z,0)^{-\alpha} }dv(w)\\
&=&2\pi^2 \int_0^\infty |z|^2 \Big(|z|^2+r(z,0)^{-\alpha}-(|z|^2-r(z,0)^{-\alpha})\Big) |z| d|z|\\
&=&4\pi^2 \int_0^\infty |z|^3  r(z,0)^{-\alpha} d|z|\\
&=&2\pi^2 \int_0^\infty r   (1+r)^{-\alpha} d r\\
&<&2\pi^2 \int_0^\infty   (1+r)^{-\alpha+1} d r\\
&=&{2 \pi^2 \over \alpha-2}.
\end{eqnarray*}
 Therefore, if
$\alpha\in (2,3)$, by symmetry in $z$ and $w$ variables, we have
$\{1, z, w\}\in A^2(D)$.
Next we claim that $z^p w^q\not\in A^2(D)$ if $p+q\ge 2$ for
$2<\a<3$. In fact,
$$
\int_{D(\alpha)} |z^p w^q|^2 dv(z,w) \gtrsim \int_{D(\alpha)}
|z^{2(p+q)}| dv(z,w) \gtrsim  \int_0^\infty  r^{p+q} (1+r)^{-\alpha}
d r=+\infty.
$$ Since $D(\a)$ is a complete Reinhardt domain,
$L^2$-integrable monomials form an orthogonal basis of $A^2(D(\a))$.
Hence,
$$
A^2(D)=\hbox{Span}\{1, z, w\}.
$$
Moreover, by symmetry,  $\|w\|_{A^2(D(\a))}=\|z\|_{A^2(D(\a))}$.

The Bergman kernel function of $D(\a)$ can thus be written as
$$
K(z,w)=c_0(1+c_1|z|^2+c_1|w|^2)\quad \hbox{for } (z,w)\in D(\alpha),
$$
where $c_0=1/v(D(\alpha))$ and $c_1=(c_0\int_{D(\alpha)}|z|^2 dv)^{-1}$ are  positive constants depending only on
$\a$.  Since it is isometric to the Fubini-Study metric of $\PP^2$
restricted to a domain in the cell $\CC^2$ with
$$
\omega_{st}= \p\dbar \log (1+|z'|^2+|w'|^2), \ \hbox{through the
scaling } (z',w')=(\sqrt{c_1}z, \sqrt{c_1}w)
$$
the Bergman metric of $D(\a)$  has a constant holomorphic sectional
curvature equal to $2$.

\end{proof}

Notice that $D(\a)$ is unbounded  and $A^2(D(\a))$ has
dimension three. On the other hand, we have Theorem \ref{thm6.1} which we restate as follows for the  reader's convenience:

\begin{theorem}\label{thm6.11}
Let $\Omega$ be a  complex manifold of dimension $n$. Suppose its
Bergman space $A^2(\O)$ is base point free and
separates holomorphic directions. Suppose that the Bergman metric of
$\O$ has positive constant holomorphic sectional curvature.  Then
$A^2(\O)$ is of finite dimension. Moreover, $\O$ is biholomorphic to
a  domain in ${\mathbb P}^n$.
\end{theorem}

\begin{proof}[Proof of Theorem \ref{thm6.11}] Let $\Omega$
be a complex manifold with its Bergman metric being well defined and
having a positive constant holomorphic sectional curvature. In a
local holomorphic chart $(U, z)$ with $z(p)=0$ for a certain $p\in
U$. Still write $U$ for $z(U)$ for simplicity of notation. Then by a
result of  Bochner  (see \cite{Bo} or \cite{He}),  there is a small
neighborhood, still denoted by $U$, of $0\in \O$ and a biholomorphic
map $F=[1,F^o]$ from $U$ to a neighborhood $V\subset \PP^n$ of
$[1,0,\cdots,0]$ with $F^o(0)=(z_1,\cdots,z_n)A+o(|z|^2)$, where $A$
is  an $n\times n$ constant matrix with $\h{det}(A)\not =0$,  such
that $ \o_\O=F^*(\ld \o_{st})$ for a certain positive constant
$\ld$. Here we use $[1,z]$ for coordinates of $V$; and  $\o_{st}$ is
the Fubini-Study metric of $\PP^n$. Let $\{\phi_j\}_{j=1}^{N}$ be an
orthonormal basis of $A^2(\O)$  and write $\phi_j=\phi_j^o
dz_1\wedge\cdots dz_n$ in $U$. We assume, without loss of
generality, that $\phi_1^o(0)\not =0$ and $\phi^o_j(0)=0$ for $j>1$.
Then
$$i\p\ov{\p}\log\sum_{j=1}^N \Big|\phi_j^o \Big|^2=i\p\ov{\p}\log(1+|F^o|^2)^\ld.$$
Then applying a standard trick, the above gives the following
identity for $z$ near $ 0$:
$$\sum_{j=2}^N \Big|\frac{\phi_j^o}{\phi_1^o} \Big|^2=(1+|F^o|^2)^\ld.$$
Now we claim that $\ld$ is a positive integer.  Indeed, after a
change of holomorphic coordinates of the form $z'=F^o(z)$ , we have
$(1+|z'|^2)^\ld=\sum_{j=2}^N
\Big|\frac{\phi_j^o}{\phi_1^o}((F^o)^{-1}(z')) \Big|^2.$ If $\ld$ is
not a natural number, then
$$
(1+|z'|^2)^\alpha
=1+\sum_{k=1}^\infty {\lambda\cdots (\lambda-k+1) \over k!} |z'|^{2k}.
$$
 This implies that a certain coefficient in the Taylor expansion
at $0$ of $(1+|z'|^2)^\ld$ is negative, equivalently, there is a
multiple index $\a$ such that
$\p^{\a}\ov{\p}^\a(1+|z'|^2)^\ld|_0<0$. On the other hand,
$\p^{\a}\ov{\p}^\a \sum_{j=2}^N
\Big|\frac{\phi_j^o}{\phi_1^o}((F^o)^{-1}(z')) \Big|^2(0)\ge 0.$
This is a contradiction. (See also \cite{HX1} for a similar
statement.)

Hence $\ld=m$ is a natural  number. Next $(\PP^n,m \o_{st})$ can be
(one to one) isometrically embedded into $\PP^{N}\subset
\PP^{\infty}$ through a Veronese map $\Psi$ for a sufficiently large
$N$. In a holomorphic chart near $[1,0,\cdots,0]$, we can write
$\Psi[1,w_1,\cdots,w_n]=[1, w_1,\cdots,w_n,\psi(w),0,\cdots]$. Here
each component in $\psi$ is a certain monomial of degree at least $2$
in $w=(w_1,\cdots,w_n)$.

Consider the holomorphic map $\Psi\circ F$ from $U$ into
${\PP}^\infty$. Then $(\Psi\circ F)^*(\o_{st})=\o_\O$ over $U$. By
the Calabi extension theorem (\cite{Ca} or \cite{HLi}), $\Psi\circ
F=[1,F^o,\psi(F^o),0,\cdots]$ extends holomorphically along any
curve of $\O$ starting from $U$.
In particular, it is easy to see that $F$ extends
holomorphically along any path in $\O$ started from $U$.

By Calabi's rigidity theorem (\cite{Ca} or \cite{HLi}), there is a
one to one isometric transformation (whose inverse is also an isometric transformation) $[L]: {\mathcal X}_{\mathcal B}\ra
{\mathcal X}_{\Psi\circ \F}$ such that
$$
 (\Psi\circ F)(z)=[L]\circ{\mathcal B}(z),\quad z\in U.
$$
 Here $[L]$ is induced by a one-to-one linear  isometric isomorphism
$L$ from the closed linear subspace $X_{(1,F^o,\psi(F^o),0,\cdots)}\subset \ell^2$
to $X_{(\phi_1,\phi_2,\cdots)}\subset \ell^2$.  Recall that $F$ extends
holomorphically and isometrically along any curve in $\O$ starting at
$z_0$. The extended map $F$ a priori might be multi-valued. However
$\Psi$  is one to one and $\mathcal B$ is globally defined over $\O$,  we see that the extension of
$\Psi\circ F$ is independent of the choice of the chosen curves and
thus gives a well defined one to one holomorphic isometry from
$(\O,\o_{\O})$ into $\PP^\infty$. Since $\Psi$ is a one to one
holomorphic isometric embedding, we conclude that $F$ extends to a
 holomorphic local isometry from $(\O,\o_\O)$ to its image
$D\subset \PP^n$.  Next we have the following

\begin{lemma}\label{one-to-one}
Assume that the Bergman metric of the complex manifold  $\O$  is locally holomorphically  isometric to the model $(\PP^n,m\o_{st})$ with  $m\in \NN$. 
Then the Bergman-Bochner map ${\mathcal B}_{\O}$ of $\O$ separates points in $\O$.
\end{lemma}

\begin{proof}[Proof of Lemma \ref{one-to-one}] Assume that ${\mathcal B}_{\O}$ does not separate points.  Then we can find $p, p^*\in \O$ such that $p\not =p^*$ but ${\mathcal B}_{\O}(p)={\mathcal B}_{\O}(p^*)$. Let $\gamma \subset \O$ 
be a smooth Jordan curve connecting $p$ to $p^*$ and let $U\Subset \O$ be a neigborhood of $\gamma([0,1])$. We can then find a sufficiently small positive constant $\d_0>0$ depending only on $U$ such that for any $p\in U, q\in \PP^n$ $v\in T_p^{(1,0)}\O, \wt{v}\in T^{(1,0)}_q{\PP^n}$, 
there is a unique  isometric biholomoprhic map
$\Phi_p$ from the ball $B_{2\d_0}(p)$ centered at $p$ with radius $2\d_0$, with respect to the metric $\o_{\O}$,  to  a ball   $\wt{B}_{2\d_0}(q))$ centered at $q$ of radius $2\delta_0$ (with respect to the metric $m\o_{st}$) with  $\Phi_p(p)=q$ and $(\Phi_p)_{*}(v)=\wt{v}$.

Denote $p$ by $p_1$, fix $v_1\in  T_{p_1} ^{(1,0)}\O$ and $\wt{v_1}\in T_{q_1}^{(1,0)}\PP^n$ with $q_1=[1,0,\cdots,0]$ and let $\Phi_1$ such the isometric biholomorphism  from $B_{2\d_0}(p_1)$ to $\wt{B}_{2\d_0}(q_1)$  such $\Phi_1(p)=q_1$  and $(\Phi_1)_{*}(v_1)=\wt{{v}_1}$. Let $t_1$ be the smallest $t$ such tha
$\g(t_1)\in \p B_{\d_0}(p_1)$. Let $p_2=\g(t_1)$, $q_2=\Phi_1(p_2)$, $v_2\in   T_{p_2}^{(1,0)}\O$ and $\wt{v_2}=(\Phi_1)_{*}(v_2)$. Define the isometric biholomorphism $\Phi_2$ as described above such  that 
$\Phi_2(p_2)=q_2$  and $(\Phi_2)_{*}(v_2)=\wt{v_2}$. 
Notice $\Phi_1\equiv \Phi_2$ in $B_{2\d_0}(p_1)\cap  B_{2\d_0}(p_2)$ by the uniqueness of isometries.  Continuing this process, we can find a finite sequence of holomorphic elements $\{B_{2\d_0}(p_j),\Phi_j\}_{j=1}^{n^*}$ with $p_{n^*}=p^*$ and 
$$n^*\le \frac{\hbox{ length of } \gamma ([0,1])}{\d_0}$$ such that $\Phi_{k}\equiv \Phi_{k+1}$ in $B_{2\d_0}(p_{k})\cap B_{2\d_0}(p_{k+1})$ for $k=1,\cdots, n^*-1$.

Write $q_j=\Phi_j(p_j)$. Then $\{\Phi_j^{-1}, \wt{B}_{2\d_0}(q_j)\}_{j=1}^{n^*}$  continues holomorphically the holomorphic element $(\Phi_1^{-1}, \wt{B}_{2\d_0}(q_1))$ to   $(\Phi_{n^*}^{-1}, \wt{B}_{q_{n^*}}(2\d_0))$ along $\wt{g}=F(\gamma)$.  {\color{blue}  By Calabi's rigidity theorem,
we have ${\mathcal B}_{\O}\circ \Phi_{1}^{-1}=[\wt{L}]\circ \Psi$ in $\wt{B}_{q_1}(2\d_0)$ where $[\wt{L}]$ is induced by a certain isometric isomorphism $\wt{L}$ from a certain closed subspace of  $\ell_2$ into another one as explained before. ADD THE DEFINITION OF $\Psi$.}
By the Calabi extension theorem or simply the uniqueness of holomorphic maps, we also obtain  ${\mathcal B}_{\O}\circ \Phi_{n^*}^{-1}=[\wt{L}]\circ \Psi$ over $\wt{B}_{2\d_0}(q_{n^*})$. Notice that  $\Phi_{1}^{-1}(q_1)=p$, $\Phi_{n^*}^{-1}(q_{n^*})=p^*$ and $q_1=q_{n^*}$, we reach a contradiction with the fact that
$\Psi$ is one-to-one.
 This proves the lemma.
\end{proof}

We continue the proof of the theorem. An immediate consequence of Lemma \ref{one-to-one} is that the map $F$ constructed above is one to one as $$
 (\Psi\circ F)(z)=[L]\circ{\mathcal B}(z),\quad z\in U.
$$
and the other three maps in the identity are one to one. Hence $F$ is a biholomorphic  map from $\O$ to its image $D\subset \PP^n$.

Next letting $E=F^{-1}\{[0,z_1,\cdots,z_n]\in D\}$ and considering
$\O\sm E$ instead of $\O$, we may assume that $D$ is a domain in
$\CC^n\subset {\PP}^n$ containing the origin.

 Since a proper complex analytic subvariety of $\O$ is a removable set for $A^2(\O)$,
  $A^2(D)$ is canonically isomorphic to $A^2(\O)$ for the original $\O$.   For simplicity of
notation, let us still write $\{\phi_j\}_{j=1}^{N}$ for an
orthonormal basis of the Bergman space $A^2(D)$ with $\phi_1(0)>0$
and $\phi_j(0)=0$ for $j\ge 2$. Notice that
$K_D(z,w)=\sum_{j=1}^{N}\phi_j(z)\ov{\phi_j(w)}$ and
$K_D(z,0)={\phi_1(0)}\phi_1(z).$ Hence, using the equation $\o_{D}=m
\o_{st}$, we similarly derive  the following:
\begin{equation}
\sum_{j=1}^{N}|\phi_j|^2=K_D(z,z)=\frac{|K_D(z,0)|^2}{K_D(0,0)}(1+|z|^2)^m=
\sum_{j=1}^{N^*}|\psi_j|^2.
\end{equation}
Here $\psi_1,\cdots,\psi_{N^*}$ are certain holomorphic functions in
$D$ with $N^*<\infty$.

By the D'Angelo lemma (see  \cite{HLi}), it follows that
each $\phi_j\in \h{Span}_{\CC}\{\psi_1,\cdots,\psi_{N^*}\}$. Hence
$N\le N^*< \infty$ or $A^2(\O)$ is of finite dimension. This
completes the proof of the theorem.
\end{proof}

We finish this section by formulating two open questions. In Theorem \ref{thm6.1}, $\Omega$ might
not be pseudoconvex. Instead, we assumed  that $A^2(\Omega)$ is of
infinite dimension. A natural question is then to ask if this
assumption can be replaced by  the Steinness of $\O$. More
precisely,   we make the following conjecture:

\begin{conjecture}\label{conj1} Let $M$ be a Stein manifold.
 Suppose that its Bergman space $A^2(M)$ is base-point
free and  separates holomorphic directions.
Then the Bergman metric of $M$ cannot have a positive constant
holomorphic sectional curvature.
\end{conjecture}
It is known by the work of Carleson and Wiegerinck \cite{Wi} that a
domain in $\CC$ either has a trivial Bergman space or an infinite
dimensional Bergman space. There is an old conjecture dating back to
Wiegerinck \cite{Wi} that asserts that the Bergman space of a
pseudoconvex domain in $\CC^n$ is either trivial or of infinite
dimension, too. (See  the work in \cite{JP} \cite{Ju} \cite{GGV} as
well as the references therein). If this old folklore conjecture
could be verified, by Theorem \ref{thm6.1}, an immediate consequence
would be the affirmative answer to Conjecture \ref{conj1}. More
generally, we propose the following:

\begin{conjecture}\label{conj2} Let $M$ be a Stein manifold of complex dimension
at least two. Suppose that its Bergman space $A^2(M)$ is base-point
free and also separates holomorphic directions.
Then the Bergman metric of $M$ cannot have  holomorphic sectional
curvatures bounded from below by a positive constant.
\end{conjecture}

\section{Bergman metric of a Stein manifold
 with negative constant holomorphic sectional curvature}

It is  known that every domain $\Omega\subset \CC^n$ with a complete
Bergman metric is pseudocovex \cite{Bre}. The converse is not true.
A simple example is a domain of the form $\BB_n \setminus V$ with
$V$ a non-empty  complex analytic variety of pure codimension one in
$\BB_n$. More generally, let $V\not =\emptyset$ be a  closed pluripolar
subset in $\O$, i.e., for any point $p\in V$, there is a neighborhood
$U_p$ of $p$ and a non-constant plurisubharmonic  function $\rho_p$ in $U_p$
such that $U_p\cap E\subset \{z\in U_p:\ \rho_p(z)=-\infty\}$. Then $A^2(\Omega \sm V)$ is naturally
identified with $A^2(\O)$ as any closed pluripolar subset is a removable
set for $A^2(\O)$. If the Bergman metric $\omega_{\Omega\sm V}$ has a
negative constant
holomorphic sectional curvature, then so does  the  Bergman metric  $\omega_\Omega$. 
Thus when $\o_{\O}$ is complete, by the classical Lu theorem
(\cite{Lu}), there is a biholomorphic map $F$ from $\O$ to $\BB^n$,
whose restriction to $\O \sm V$ gives a biholomorphic isometry
from $(\Omega\sm V, \omega_{\Omega\sm V})$ to $(\BB_n\setminus V^0,
\omega_{\BB_n\setminus V^0})$, where $V^0=F(V)$ is a pluripolar
subset in $\BB^n$. Notice that $(\BB^n\setminus V^0,
\omega_{\BB^n\setminus V^0})$ is incomplete along $V^0$.

Hence, we have many pseudoconvex domains whose Bergman metrics have negative constant holomorphic sectional curvature. 
In this section, we prove  these are the only  examples, namely, we prove Theorem \ref{thm5.2}.
 As we mentioned in the introduction, special cases of this theorem were proved
 earlier  with a different method in the work of
Dong-Wong \cite{DW1} \cite{DW2}. For the reader's convenience we restate Theorem \ref{thm5.2} as follows:

\begin{theorem}\label{thm5.22}
Let $\Omega$ be a Stein manifold of complex dimension $n\ge 1$.
Assume its Bergman space is base point free and
separates holomorphic directions.
 Then the Bergman of $\Omega$  has a negative constant holomorphic sectional
curvature if and only if there is a biholomorphic  map $F$ from $\Omega$ to  $\BB^n\sm E$, where $E$ is
a   closed pluripolar subset of $\BB^n$.
\end{theorem}

\begin{proof}[Proof of Theorem \ref{thm5.22}]
Let $\{\phi_j\}_{j=1}^{N}$ be an orthonormal basis of $A^2(\O)$.
Under the hypothesis in the theorem,  the corresponding
Bergman-Bochner map $\F$   is a local  holomorphic isometric
embedding from $\O$ into $ \PP^\infty$. Here $F=(\phi_1,\cdots,
\phi_N, 0,\cdots)$ when $N<\infty$ or $F=(\phi_1,\cdots,
\phi_m,\cdots)$ when $N=\infty$. Moreover, ${\mathcal F}=[F].$

Suppose that the holomorphic sectional curvature of $\o_\O$ is a
negative constant on $\O$. Then for any $p\in \O$, there is a biholomorphic map $f$
from a small neighborhood  $U$ of $p\in \O$ into a certain domain in
the unit ball $\BB^n$ in $\CC^n$ such that
$f^*(\ld\o_{\BB^n})=\o_\O.$ Here $\ld>0$ is a certain positive
constant depending on the holomorphic sectional curvature of
$(\O,\o_\O)$ and $\o_{\BB^n}$ is the Bergman metric of $\BB^n$.

Notice that the Bergman metric of $\BB^n$ is given by
$\omega_{\BB^n}=-(n+1) i\partial \Ol{\partial} \log (1-|w|^2)$,
where $w=(w_1, \cdots, w_n)$ is the coordinates for $\mathbb{B}^n
\subset \mathbb{C}^n.$  We use the Taylor expansion of
$(1-x)^{-\mu}$ with $\mu=(n+1)\ld$ at  $x=0$ on $\{ x \in
\mathbb{R}: |x| < 1\}$ to get
$$(1-|w|^2)^{-\mu}= 1+ \sum_{k=1}^{\infty} \frac{\mu(\mu+1) \cdots (\mu+k-1)}{k!} |w|^{2k}.$$
The right hand side converges uniformly on compact subsets of
$\mathbb{B}^n.$ This shows that for a certain sequence of monomials
$\{P_j \}_{j=1}^{\infty}$ in $w$ we have
$(1-|w|^2)^{-\mu}=1+\sum_{j=1}^{\infty}|P_j(w)|^2$, which
converges uniformly on compact subsets of $\mathbb{B}^n.$ This leads
to a natural holomorphic isometric embedding $\mathcal{T}=[T]=[1,
P_1, \cdots, P_j, \cdots]$ from $(X_0, \mu\omega_{\BB^n})$ to
$\mathbb{P}^{\infty}.$ It is clear that $\mathcal T$ is one to one
as each $w_j$ for $1 \leq j \leq n$ (with an appropriate
coefficient) is among the $P_j'$s.

Now  ${\mathcal T}\circ f: U\to \PP^\infty$ is a one to one
holomorphic isometric embedding. By Calabi's rigidity theorem
(\cite{Ca} or \cite{HLi}), there is a one to one isometric
isomorphism $[L]: {\mathcal X}_\F\ra {\mathcal X}_{\mathcal T}=\PP^\infty$ such that
$$
[L]\circ \F(z)={\mathcal T}\circ f(z),\quad z\in U.
$$
 Notice that $[L]$ is induced by a linear isometric isomorphism
$T$ from the closed linear subspace $X_{F}$ to $X_{T}$. By the
Calabi extension theorem (\cite{Ca} or \cite{HLi}), ${\mathcal
T}\circ f$ extends along any curve in $\O$ initiated from $z_0\in U$
to a local holomorphic isometric embedding. Since each $w_j$ is a
component in ${\mathcal T}$ up to a coefficient, we see that $f$
extends holomorphically and isometrically along any curve in $\O$
starting at $z_0$ to $\wt{f}$.  $\wt{f}$ a priori might be
multi-valued. Since the normalized Bergman metric $\ld\o_{\BB^n}$ is
complete, $\wt{f}$ has image in $\BB^n$.
By the uniqueness of holomorphic functions, we have
$[L]\circ\F(z)={\mathcal T}\circ\wt{f}$ for any extended map
$\wt{f}$ of $f$ through a curve starting at $z_0$.  By a completely analogical argument as in Lemma \ref{one-to-one},  ${\mathcal F}$ is one to one. Since  $\mathcal
T$ is one to one, we see that $\wt{f}$ is independent of the choice
of the chosen curves and thus gives a well defined injective  local holomorphic
isometry from $(\O,\o_{\O})$ into $(\BB^n,\ld\o_{\BB^n})$. By the
uniqueness of real analytic functions, $\wt{f}^{*}(\ld
\o_{\BB^n})=\o_{\O}$.
In
particular, $D=\wt{f}(\O)$ is a subdomain of $\BB^n$.   Notice that
Bergman metrics are biholomorphic invariants. One arrives  at  the
following conclusion:

{\it  There is a biholomorphic map $\wt{f}: (\Omega, \o_\Omega) \to
(D, \ld\o_{\BB^n})$ where $D$ is a certain domain in $\BB^n$ such
that ${\wt f}^*(\ld \o_{\BB^n})=\o_{\O}$ for a certain positive
constant $\ld$. Moreover $\o_{D}=\ld \o_{\BB^n}|_{D}$ over $D$. }


Still write $\omega_{D}$ for the Bergman metric of $D$. Since
Bergman metrics are biholomorphic invariant, applying an
automorphism of $\BB^n$, we assume that $0\in D$. For simplicity of
notation,  we now write $\{\phi_j\}_{j=1}^\infty$ for an orthonormal
basis of $A^2(D)$ with $\phi_1(0)>0$ and $\phi_j(0)=0$ for $j\ge 2$.
Since $K_D(z,w)=\sum_{j=1}^{\infty}\phi_j(z)\ov{\phi_j(w)},$ we
conclude that $K_D(z,0)={\phi_1(0)}\phi_1(z).$

Since $\omega_{D}=\ld\omega_{\BB^n}$ on $D$, we similarly have
\begin{equation}\label{4.2}
\log K_D(z,z)-\log|\phi_1(z)|^2=\log \big(K_{\BB^n}(z,z)\big)^\ld\
\h{for }z\in D.
\end{equation}
Hence
\begin{equation}\label{4.4}
K_{D}(z, w)=|\phi_1(z)|^2 K^\ld_{\BB^n}(z,w) \quad z, w\in D.
\end{equation}
Therefore
\begin{equation}\label{4-5}
K_{D}(z,w)= c K_{D} (z,0) K_{D}(0,w) K^\ld_{\BB^n}(z,w)\ \ \h{with}
\ \ c=\frac{1}{{\phi^2_1(0)}}=\frac{1}{{K_{D}(0,0)}}.
\end{equation}
For any element $h\in A^2(D)$, one has
$$
h(z)=c K_{D}(z,0) \int_{D} K^\ld_{\BB^n}(z, w) K_{D}(0,w) h(w)
dv(w).
$$
As before, we write $dv(w)$ for the Lebegue measure of $\CC^n$ with
coordinate $w$. In particular,
$$
1=c K_{D}(z,0)  \int_{D} K^\ld_{\BB^n}(z,w) K_{D}(0, w) dv(w).
$$
Therefore
$$
{1 \over K_{D}(z, 0)}= c \int_{D} K^\ld_{\BB^n}(z,w) K_{D}(0, w)
dv(w)
$$
extends  holomorphically to $ \BB^n$, as $K^\ld_{\BB^n}(z,w)$ is an
anti-holomorphic function over $\ov{\BB^n}$ in $w$,   holomorphic in
$z\in \BB_n$ and $K^\ld_{\BB^n}(z,w)$ is bounded in $w\in \BB^n$ for
any fixed $z\in \BB^n$.

Let $f_0(z)={1\over K_D(z, 0)}$. Then $f_0(z)$ is holomorphic in
$\BB_n$ and $f_0(z)\not\equiv 0$. Let $Z(f_0)=\{z\in \BB^n:
f_0(z)=0\}$. Then for any $h\in A^2(D)$ one has that
\begin{equation}\label{4.5}
h(z)=c K_{D}(z,0) \int_{D} K^\ld_{\BB^n}(z, w) K_{D}(0,w) h(w) dv(w)
\end{equation}
extends to be a holomorphic function  in $\BB^n\setminus
Z(f_0)$.




Write $\p D\cap \BB^n=L_1\cup L_2$ with
  $L_1=\ \p D\cap Z(f_0)$ and $L_2=\p D\cap \BB^n \setminus Z(f_0)$.
  For  $z_0\in L_2$, by (\ref{4-5}), one has
  \begin{equation}
  \lim\sup_{z\in D, z\to z_0}  K_{D}(z_0, z_0)<\infty.
  \end{equation}
  Since $D$ is pseudoconvex,
  by a result of Pflug and Zwonek (Lemma 11 of \cite{PZ}) that was
   based on a one dimensional classical result
  and the Ohsawa-Takegoshi extension theorem, for each $z_0$ in $\BB^n\sm L_1$
   there is a neighborhood  $U_{z_0}$ of
  $z_0$  such that $P_{z_0}=U_{z_0}\setminus D$ is
 a pluripolar set. Since $P_{z_0}$  has Hausdorff codimension at
 least two, $U_{z_0}\sm  P_{z_0}$ is connected.
 Thus $P_{z_0}=\p D\cap U_{z_0}$
 and  $D\cap U_{z_0}=U_{z_0}\sm P$. Moreover, $\overline{D}=\overline{\BB^n}$.
 Now by the Josefson theorem \cite{Jo}, there is a
 plurisubharmonic function  $\rho$ defined in $\CC^n$ such that
 $L_2\subset \{\rho=-\infty\}$. Then $E:=\p
 D \cap \BB^n\subset \{\rho+\log|k_D(z,0)|=-\infty\}$.  Namely, $\p D \cap \BB^n$ is a
 pluripolar set.
 Again since a (closed) pluripolar set has Hausdorff codimension at
 least two, $\BB^n\sm \p D$ is connected which then has to be $D$ as
 $D$ is assumed to be connected. Namely, $D=\BB^n\sm E$ with $E$ is relatively closed in $\mathbb{B}^n$.
 Since any relatively closed pluripolar set  is a removable set  for $L^2$-integrable
holomorphic functions (Lemma 1, \cite{Ir}). So, the Bergman metric of $D$
is just the restriction of $\o_{\BB^n}$  on  $D$. Hence
$\ld=1$.
 This proves
 the theorem.
\end{proof}


\section{Proof of Theorem \ref{thm7.1}}



\begin{proof}[Proof of Theorem \ref{thm7.1}]
Assume that $ (\O, \omega_\Omega) $ is as in the theorem. Let
$(U,z)$ with $0\in z(U)$ be a holomorphic coordinate chart. (Still
write $U$ for $z(U)$).
Shrinking $U$ if needed, there is a  biholomorphic map  $F$ from $U$
to a neighborhood $V\subset \CC^n$ of $0$ such that $F(0)=0$ and
$F^*(\o_{eucl})= \o_\O,$ where $\o_{eucl}=i\p\ov{\p}|z|^2$ is the
Euclidean metric of $\CC^n$. Then by a similar argument as before,
 we conclude that $F$ extends
to a  biholomorphic map from $\O$ to its image $D\subset \CC^n$ with
$F^*(\o_{eucl})=\o_\O.$ Hence
$$\o_D=\o_{eucl}=i\p\ov{\p}\log(e^{|z|^2}).$$

Still let  $\{\phi_j\}_{j=1}^\infty$ be an orthonormal basis of the
Bergman space $A^2(D)$ with $\phi_1(0)>0$ and $\phi_j(0)=0$ for
$j\ge 2$. Then, as before,
$K_D(z,w)=\sum_{j=1}^{\infty}\phi_j(z)\ov{\phi_j(w)}$ and
$K_D(z,0)={\phi_1(0)}\phi_1(z).$ We thus have
\begin{equation}\label{new-01}
K_D(z,z)=|\phi_1(z)|^2e^{|z|^2}=\frac{|K_D(z,0)|^2}{K_D(0,0)}e^{|z|^2}.
\end{equation}
 Therefore,
by complexification, one has
\begin{equation}\label{4.20}
K_D(z,w)={K_D(z,0) K_D(0,w)\over K_D(0,0)}e^{\langle z, w\rangle}=
\phi_1(z)
\ov{ \phi_1(w)} e^{\langle z, w\rangle},\quad \hbox{for } z, w\in D,
 \end{equation}
where
\begin{equation}
\phi_1(z)=K_D(z,0) K_D(0,0) ^{-1/2}\quad \h{for } z\in D.
\end{equation}
Since $D$ is open and $0\in D$, there is an $r>0$ such that, the
closed ball with radius $r$, $\overline{B(0, r)}\subset D$. Notice
that
\begin{equation}\label{4.23}
|z_j|^{2m} \le C_{m} \sum_{k=0}^5 e^{r \Re (e^{i k \pi/3} z_j)}\ \
\quad \h{for }\  z\in D \ \h{and } \ 1\le j\le n
\end{equation}
and
\begin{equation}\label{4.24}
 \int_D |e^{\langle z, w\rangle } \phi_1(w)|^2 dv(w)  = {1\over |\phi_1(z)|^2}
 \int_D |K_D(z, w)|^2 dv(w)=e^{|z|^2} \quad \h{for } \ z\in D.
\end{equation}
From (\ref{4.23}) and (\ref{4.24}), it follows that
\begin{equation}\label{4.25}
\int_D |w_j^m \phi_1(w)|^{2} dv(w)\le 6 C_m e^{r^2}< \infty
\quad\hbox{for any } 1\le j\le n \hbox{ and } m\in \NN.
\end{equation}
This shows that $w^\alpha \phi_1(w) \in A^2(D)$ for each multi-index 
$\alpha$.  Therefore,
$$
z^\alpha\phi_1(z) =\int_D K_D(z,w) w^\alpha \phi_1(w) dv(w)=
\phi_1(z) \int_D e^{\langle z, w\rangle} w^\alpha  |\phi_1(w)|^2
dv(w).
$$
In particular,
\begin{equation}\label{4.26}
z^\alpha= \int_D e^{\langle z, w\rangle} w^\alpha |\phi_1(w)|^2
dv(w).
\end{equation}
Let $\{D_k\}_{k=1}^\infty$ be an increasing sequence of
smoothly bounded domains in  $D$ with $\bigcup_{k}D_k=D$. Write
$$
J_k^\alpha (z) = \int_{D_k} e^{\langle z, w\rangle} w^\alpha |\phi_1(w)|^2
dv(w)\quad \hbox{and} \quad J^\alpha (z)=\int_{D} e^{\langle z, w\rangle}
w^\alpha |\phi_1(w)|^2 dv(w)=z^\a.
$$
Notice that $\int_{D_k}  |e^{\langle z, w\rangle} \wbar^\beta w^\alpha
\phi_1(w)^2| dv(w)$ is   uniformly bounded  for  $z$ in  any compact subset of  $D$ by
(\ref{4.24}) and (\ref{4.25}) and notice that  $J^\alpha _k$ is holomorphic in $D$.
Moreover, by (\ref{4.26}) and the Lebesgue dominated convergence theorem, we have
$$J^\alpha_k(z)\to
\int_D e^{\langle z, w\rangle} w^\alpha |\phi_1(w)|^2
dv(w)=J^\alpha (z)=z^\alpha.
$$
Hence, by passing to a subsequence if needed,  we see that  $J^\alpha_k(z)$ converges to $z^\a$  uniformly on
compact subsets in $D$. Since ${\p^{|\beta|} J^\alpha_k(z)\over \p
z^\beta}=\int_{D_k} e^{\langle z, w\rangle} \wbar^\beta w^\alpha
|\phi_1(w)|^2 dv(w) $, we have  
$$
{\p^{|\beta|} J^\alpha (z)\over \p z^\beta} =\lim_{k\to\infty} {\p^{|\beta|}
J^\alpha_k(z)\over \p z^\beta},\quad z\in D.
$$
In particular, at $z=0$, one has
\begin{equation} \label{Eq1}
\alpha! \delta_{\alpha, \beta} ={\p^{|\beta|}  z^\alpha \over  \p
z^\beta }|_{z=0}={\p^{|\beta|} J^\alpha (z) \over \p z^\beta} 
|_{z=0} =\int_D \wbar^\beta w^\alpha |\phi_1(w)|^2 dv(w).
\end{equation}
This shows that $\{{z^\alpha \over \sqrt{\alpha!}}
\phi_1\}_{|\alpha|=0}^\infty$ forms an orthonormal set for $A^2(D)$.
Notice that
\begin{equation}
\sum_{|\alpha|=0}^\infty {z^\alpha \over \sqrt{\alpha!}
}{\wbar^{\alpha}\over \sqrt{\alpha!}}= \sum_{|\alpha|=0}^\infty
{z^\alpha \wbar^\alpha\over \alpha!}
=\sum_{\alpha_1=0}^\infty {z_1^{\alpha_1} \wbar^{\alpha_1}\over
{\alpha_1}!}\cdots \sum_{\alpha_n=0}^\infty {z_n^{\alpha_n}
\wbar^{\alpha_n}\over \alpha_n!} =e^{\langle z, w\rangle}.
\end{equation}
Therefore,
\begin{equation}
\sum_{|\alpha|=0}^\infty {z^\alpha \over \sqrt{\alpha!} } \phi_1(z)
{\wbar^{\alpha}\over \sqrt{\alpha!}}\phi_1(w)= e^{\langle z,
w\rangle} \phi_1(z)\phi_1(w)=K_D(z,w).
\end{equation}
Hence $\{ {z^\alpha\over \sqrt{\alpha!}}\phi_1\}$ forms an
orthonormal basis for $A^2(D)$ and $\phi_1\not =0$. 
\end{proof}

We remark that an immediate  consequence of Theorem \ref{thm7.1} is that the Bergman metric of $\O$ in Theorem \ref{thm7.1} cannot be flat when $\O$ admits a non-trivial  bounded holomorphic function. (See the earlier   version of  this paper posted  at  arXiv:2302.13456 ). However, in the Appendix provided by John Treuer,    Theorem \ref{thm7.1}, together with results from the complex  moment theory,  will show that the Bergman metric of $\O$ cannot be flat  in general.

We
construct in what follows an unbounded pseudoconvex Hartogs domain
$D\subset \CC^{n+1}$ which contains a copy of $(\CC^n, \o_{eucl})$
as a totally geodesic complex submanifold of $(D,\o_D)$.

\medskip

 Let $D$ be  the Reinhardt and Hartogs domain defined by
\begin{equation}
D=\{ z=(z, w)\in \CC^n\times \CC: |w|^2 < e^{-|z|^2} \}.
\end{equation}
Notice that $D$ is pseudoconvex as its defining function
$\rho=2\log|w|+|z|^2$ is plurisubharmonic along its boundary. Also
notice that $h(z,w)=w$ is a non-constant bounded holomorphic
function over $D$. We will show in what follows that $(\CC^n\times
\{0\}, \o_{eucl})$ is a total geodesy of $(D,\o_D)$. Namely, the
holomorphic  embedding $\iota: (\CC^n,\o_{eucl})\ra (D,\o_D)$ with
$\iota(z)=(z,0)$ is an isometric embedding. Moreover, the second
fundamental form of $\CC^n\times \{0\}$ in $D$ is identically zero,
which is equivalent to the statement that  holomorphic sectional
curvatures of $(D,\o_D)$ are zero along $\CC^n\times \{0\}$ by the
Gauss-Codazzi equation.

\medskip
We first prove that the Bergman space $A^2(D)$ is of infinite
dimension because  $\{z^\alpha w^q\}_{\a\ge 0, q\ge 0}$ forms an
orthogonal basis of $A^2(D)$. To this aim,  it suffices prove  that
\begin{equation}\label{moment-001}
\int_D|z^\alpha w^q|^2 dv={\pi^{n+1}\over (q+1)^{|\alpha|+n+1}}\alpha!<\infty,
\end{equation}
as $D$ is a complete Reinhardt domain. Indeed,
\begin{eqnarray}\label{Eq2}
\|z^\alpha w^q\|^2_{L^2}&=&\int_D |z^\alpha|^2 |w|^{2q} dv \nonumber\\
&=&\int_{\CC^n} |z^\alpha|^2 dv(z) \int_{|w|^2< e^{-|z|^2} } |w|^{2q} dv(w)\nonumber\\
&=&\pi \int_{\CC^n} |z^\alpha|^2 {1\over q+1} e^{-(q+1)|z|^2} dv(z)\nonumber\\
&=&{\pi\over (q+1)^{|\alpha|+n+1}}  \int_{\CC^n} |z^\alpha|^2 e^{-|z|^2} dv(z)\\
&=&{\pi^{n+1}\over (q+1)^{|\alpha|+n+1}}\prod_{j=1}^n \int_0^\infty r_j^{\alpha_j}
e^{-r_j} dr_j \nonumber \\
&=&{\pi^{n+1}\over (q+1)^{|\alpha|+n+1}}\alpha!.\nonumber
\end{eqnarray}
Therefore, the diagonal  Bergman kernel function $K$ of $D$  is
computed as follows:
\begin{eqnarray*}
K &=&\sum_{|(\alpha, p)|=0}^\infty {1\over \|z^\alpha w^p\|^2_{L^2(D)} } |z^\alpha|^2 |w|^{2p }\\
&=&{1\over \pi^{n+1}} \sum_{|\alpha|=0}^\infty {|z^\alpha|^2 \over \alpha!} \sum_{p=0}^\infty (p+1)^{|\alpha|+n+1} |w|^{2p} \\
&=&{1\over \pi^{n+1}}\sum_{p=0}^\infty  (p+1)^{n+1}  \sum_{|\alpha|=0}^\infty {|z^\alpha|^2 \over \alpha!} (p+1)^{|\alpha|} |w|^{2p} \\
&=&{1\over \pi^{n+1}}\sum_{p=0}^\infty  (p+1)^{n+1} e^{(p+1)|z|^2} |w|^{2p}\\
&=&{ e^{|z|^2} \over \pi^{n+1}}\sum_{p=0}^\infty  (p+1)^{n+1} \Big(
e^{|z|^2} |w|^{2}\Big)^p.
\end{eqnarray*}

When $n=1$,
$$
(x({x\over 1-x})')'={1+x\over (1-x)^3} =\sum_{p=0}^\infty (p+1)^2
x^p
$$
and thus
$$
K={e^{\|z\|^2}\over \pi^2} {1+e^{|z|^2}|w|^2 \over ( 1-e^{|z|^2}
|w|^2)^3}.
$$

For a general $n$, when $w=0$, one has
\begin{equation}
K_0={e^{|z|^2}\over \pi^{n+1}}\quad \h{for } z\in \CC^n,\h{ and }
K=K_0(1+|w|^2h(z,w)) \h{ for } (z,w)\in D,
\end{equation}
where $h(z,w)$ is a  real analytic function in $D$. Write $z_{n+1}$
for $w$ in what follows and write the Bergman metric tensor 
\begin{equation}
g_{\a\ov\b}=\frac{\p^2 \log K}{\p z_\a{\p  \ov z_\b}}=\frac{\p^2
\left(\log K_0+O(1)|w|^2\right)}{\p z_\a{\p  \ov z_\b}}
\end{equation}
with $\a,\b\in \{1,\cdots,n+1\}$. Clearly, $g_{(n+1)\ov
j}=g_{j\ov{(n+1)}}=0$ for $1\le j\le n$. Also $\frac{\partial g_{i
\ov{i}}}{\p z_k}
 =\frac{\partial g_{i\ov j}}{{\p
 \ov z_l}}=0$ when $w=0$ and $1\le i,j,k\le n$.  Hence when $w=0$, $g^{j\ov{(n+1)}}=g^{(n+1)
 \ov j}=0$ for $j\le n$. Now for $1\le i,j,k,l\le n$ and $w=0$, the Riemannian curvature
tensions $R_{i\ov{j}k\ov{l}}$ are computed as follows:
 \begin{equation}
 R_{i\ov{j}k\ov{l}}=-\frac{\partial^2 g_{i\ov j}}{\p z_k{\p
 \ov z_j}}+\sum_{\a,\b=1}^{n+1}g^{\a\ov\b}\frac{\partial g_{i\ov \b}}{\p z_k}
 \frac{\partial g_{\a\ov j}}{{\p
 \ov z_l}}=0.
 \end{equation}
 We thus conclude that when $w=0$,  holomorphic sectional curvatures of the Bergman metric of $D$ along
 directions tangent to $\CC^n\times \{0\}\subset D$ are zero.

By the Gauss-Godazzi equation for K\"ahler submanifolds (see, e.g.,
[page 33, \cite{Mo1}]), we see that the second fundamental form of
$\CC^n\times \{0\}$ in $D$ is zero.  Thus, $\CC^n\times \{0\}$ with
the flat metric is a totally geodesic submanifold in $(D,\o_D)$.

\bigskip
\bigskip

\noindent Xiaojun Huang (huangx$@$math.rutgers.edu): Department of
Mathematics, Rutgers University, New Brunswick, NJ 08903, USA

\noindent Song-Ying Li (sli$@$uci.edu): Department of Mathematics,
University of California, Irvine,  Irvine, CA 92697, USA

\eject
\vfill

\centerline {\bf  Appendix on the Proof of Theorem \ref{thm7.2}}

\medskip
\centerline{ by John N. Treuer}

In this appendix, we will provide a proof of Theorem \ref{thm7.2}. Together with Theorem \ref{thm7.1}, this will prove Corollary  \ref{thm7.3}. One of the main tools we will use is a uniqueness theorem for the moment problem.

Let $s:=(s_{\alpha \bar{\beta}})_{\alpha, \beta \in \mathbb{N}_0^n}$ be a sequence of  complex numbers with $s_{\alpha \bar{\a}}\ge 0$ for each $\a$.  
The (complex) moment problem on  $\mathbb{C}^n$ asks whether there exists a unique nonnegative Radon measure $\mu$ such that

\begin{equation}\label{complex moment problem}
s_{\alpha \overline{\beta}} = \int_{\mathbb{C}^n} z^{\alpha}\overline{z}^{\beta} d\mu(z)
:=L^{\mu}[z^\alpha \zbar^\beta], \quad \alpha, \beta \in \mathbb{N}_0^n.
\end{equation}

The moment problem is a classical problem in probability and analysis, see \cite{Sch} for more details. We will use the following uniqueness result, adapted from \cite[Theorem 15.11, p. 391]{Sch}, for the complex moment problem.  
 
\begin{theorem}\label{???-01} Suppose that there are two nonnegative Radon measures $d\mu_1$  and $d\mu_2$ such that
 $$
 L^{\mu_1}[z^\a \zbar^{\beta}]= L^{\mu_2}[z^\a \zbar^{\beta}]=\sqrt{\a!\b!}\d_{\a \b}
 \quad\hbox{ for any } \a,\b\in \mathbb{N}_0^n. 
 $$
 Then $d\mu_1= d\mu_2$.
\end{theorem}

\begin{proof}[Proof of Theorem \ref{thm7.2}]
Suppose that $D$ exists and $A^2(D)$ has an orthonormal basis of the form  $\left\{{z^{\alpha} \over \sqrt{\alpha!}}\phi(z)\right\}_{\alpha \in (\mathbb{N}_0)^n}$ with $\phi\not =0$ and 
$\phi\in A^2(D)$ with $\|\phi\|_{L^2(D)} = 1$. 
Define a nonnegative Radon measure on $\mathbb{C}^n$ by
\begin{equation*}
d\mu_1(z)=|\phi(z)|^2 \chi_{D}(z)dv(z),
\end{equation*}
where $\chi_D$ denotes the characteristic function of $D$. 
 By the orthonormality of the basis, 
\begin{equation}\label{M1}
L^{\mu_1}[z^\alpha \zbar^\beta]=\int_{\mathbb{C}^n} z^\alpha \zbar^\beta d\mu_1(z)=\int_{D} z^{\alpha}\overline{z}^{\beta}|\phi(z)|^2 dv(z) =
\alpha!\delta_{\alpha\beta}.
\end{equation}
 Define another positive Radon measure on $\mathbb{C}^n$ by
 \begin{equation*}
 d\mu_2(z)= \pi^{-n} e^{-|z|^2} dv(z).
 \end{equation*}
 By integrating in polar coordinates and using the definition of the Gamma function,
 \begin{eqnarray}
L^{\mu_2}[z^\alpha \zbar^\beta] = \int_{\mathbb{C}^n} z^\alpha \zbar^\beta d\mu_2(z)
 &=& 2^n\delta_{\alpha\beta}\int_{(0, \infty)^n} e^{-|z|^2}\prod_{k=1}^n|z_k|^{2\alpha_k + 1}d(|z_1|,\ldots,|z_n|)  \nonumber 
 \\
 &=& \alpha!\delta_{\alpha\beta}. \label{m1}
 \end{eqnarray}
By (\ref{M1}), (\ref{m1}) and Theorem \ref{???-01},
$$
|\phi(z)|^2\chi_D (z) dv(z) = \pi^{-n}e^{-|z|^2}dv(z).
$$
One then easily concludes that
$$
|\phi(z)|^2 \equiv \pi^{-n}e^{-|z|^2}, \quad z \in D.
$$
Taking the logarithm of both sides,
$$
2\ln|\phi(z)| = \ln(\pi^{-n}) - |z|^2, \quad z \in D.
$$
Applying the Laplace operator to both sides yields that
$
0=-4n$. This is a contradiction; therefore, $D$ does not exist.
\end{proof}

\medskip
\noindent John N. Treuer (jtreuer$@$ucsd.edu), Department of Mathematics, University of California, La Jolla, CA 92093.

\end{document}